\documentclass[a4paper, 11pt]{article}
\usepackage[T2A,T1]{fontenc}
\usepackage[utf8]{inputenc}
\usepackage[german, english]{babel}
\usepackage{amsmath, amsthm, amssymb,amsfonts}
\usepackage{mathtools}
\usepackage{mathrsfs}
\usepackage{enumerate}
\usepackage{graphicx}
\usepackage{color}
\usepackage{booktabs}

\usepackage{thumbpdf}
\usepackage{hyperref}

\newtheorem{theorem}{Theorem}
\newtheorem{corollary}[theorem]{Corollary}
\newtheorem{lemma}[theorem]{Lemma}

\hypersetup{
	pdftitle    = {Solving the Linear 1D Thermoelasticity Equations with Pure Delay},
	pdfsubject  = {Thermoelasticity with pure delay},
	pdfauthor   = {Denys Ya. Khusainov, Michael Pokojovy},
	pdfkeywords = {thermoelasticity; partial differential equations with delay; well-posedness; small parameter asymptotics; solution representation}
}

\begin{document}
\title{Solving the Linear 1D Thermoelasticity Equations with Pure Delay}
\author{Denys Ya. Khusainov\thanks{Department of Cybernetics, Kyiv National Taras Shevchenko University, 64 Volodymyrska Str, 01601 Kyiv, Ukraine, e-mail: \texttt{d.y.khusainov@gmail.com}} \and
Michael Pokojovy\thanks{Department of Mathematics and Statistics, University of Konstanz, Universitaetstr 10, 78457 Konstanz, Germany, e-mail: \texttt{michael.pokojovy@uni-konstanz.de}}}

\date{\today}

\maketitle

\begin{abstract}
	\noindent
	We propose a system of partial differential equations with a single constant delay $\tau > 0$
	describing the behavior of a one-dimensional thermoelastic solid occupying a bounded interval of $\mathbb{R}^{1}$.
	For an initial-boundary value problem associated with this system, 
	we prove a global well-posedness result in a certain topology under appropriate regularity conditions on the data.
	Further, we show the solution of our delayed model to converge to the solution of the classical equations of thermoelasticity as $\tau \to 0$.
	Finally, we deduce an explicit solution representation for the delay problem. \\[0.5cm]
\end{abstract}

\pagestyle{myheadings}
\thispagestyle{plain}
\markboth{\textsc{D.~Ya. Khusainov, M. Pokojovy}}{\textsc{Solving the Linear 1D Thermoelasticity Equations with Pure Delay}}

\addcontentsline{toc}{section}{Introduction}
\section*{Introduction}
	Over the past half-century, the equations of thermoelasticity have drawn a lot of attention
	both from the side of mathematical and physical communities.
	Starting with the late 50-s and early 60-s of the last century,
	the necessity of a rational physical description for elastic deformations of solid bodies
	accompanied by thermal stresses motivated the more prominent mathematicians, physists and engineers 
	to focus on this problem (see, e.g., \cite{Ca1972}, \cite{Cha1960}, etc.).
	As a consequence, many theories emerged, 
	mainly in the cross-section of (non-linear) field theory and thermodynamics,
	making it possible for the equations of thermoelasticity
	to be interpreted as an anelastic modification of the equations of elasticity (cf. \cite{Day1985} and the references therein).
	Both linear and nonlinear models and solution theories were proposed.

	An initial-boundary value problem for the general linear equations of classical thermoelasticity in a bounded smooth domain $\Omega \subset \mathbb{R}^{n}$
	\begin{align}
		\rho \partial_{tt} u_{i} = (C_{ijkl} u_{k, l})_{, j} - (m_{ij} T)_{, j} + \rho f_{i}
		 &\text{ in } \Omega \times (0, \infty), 
		\label{DAFERMOS_THERMOELASTICITY_PDE_1} \\
		\rho c_{D} \partial_{t} T + m_{ij} \theta_{0} \partial_{t} u_{i, j} = (K_{ij} T_{, j})_{, i} + \rho c_{D} r 
		&\text{ in } \Omega \times (0, \infty)
		\label{DAFERMOS_THERMOELASTICITY_PDE_2}
	\end{align}
	was studied by Dafermos in \cite{Da1968}.
	Here, $[u_{i}]$ and $T$ denote the (unknown) displacement vector field and the absolute temperature, respectively.
	Further, $\rho > 0$ is the material density, $\theta_{0}$ is a reference temperature rendering the body free of thermal stresses,
	$c_{D}$ is the specific heat capacity,
	$[C_{ijkl}]$ stands for the Hooke's tensor, $[m_{ij}]$ is the stress-temperature tensor,
	$[K_{j}]$ is the heat conductivity tensor, $[f_{i}]$ represents the specific external body force
	and $r$ is the external heat supply.
	Under usual initial conditions, appropriate normalization conditions to rule out the rigid motion as a trivial solution 
	and general boundary conditions
	\begin{align}
		u_{i} &= 0 \text{ in } \Gamma_{1} \times (0, \infty), &
		\big(C_{ijkl} u_{k, l} - m_{ij} T\big) n_{j} + A_{ij} u_{j} &= 0 \text{ in } \Gamma_{1}^{c} \times (0, \infty),
		\label{DAFERMOS_THERMOELASTICITY_BC_1} \\
		T &= 0 \text{ in } \Gamma_{2} \times (0, \infty), &
		\big(K_{ij} T_{, j}\big) n_{i} + B T &= 0 \text{ in } \Gamma_{2}^{c} \times (0, \infty)
		\label{DAFERMOS_THERMOELASTICITY_BC_2}
	\end{align}
	where $\Gamma_{1}, \Gamma_{2} \subset \partial \Omega$ are relatively open,
	$[A_{ij}]$ denotes the ``elasticity'' modulus and $B$ is heat transfer coefficient,
	Dafermos proved the global existence and uniqueness of finite energy solutions and
	studied their regularity as well as asymptotics as $t \to \infty$.
	In 1D, even an exponential stability result for Equations (\ref{DAFERMOS_THERMOELASTICITY_PDE_1})--(\ref{DAFERMOS_THERMOELASTICITY_PDE_2})
	under all ``reasonable'' boundary condition was shown by Hansen in \cite{Ha1992}.

	In his work \cite{Sle1981},
	Slemrod studied the nonlinear equations of 1D thermoelasticity in the Lagrangian coordinates
	\begin{align}
		\partial_{tt} u = \hat{\psi}_{FF}(\partial_{x} u + 1, \theta + T_{0}) \partial_{xx} u + \hat{\psi}_{FT}(\partial u_{x} + 1, \theta + T_{0}) \partial_{x} \theta &\text{ in } (0, 1) \times (0, \infty),
		\label{SLEMROD_THERMOELASTICITY_PDE_1} \\
		\rho(\theta + T_{0}) \big(\hat{\psi}_{TT}(\partial_{x} u + 1, \theta + T_{0}) \partial_{t} \theta +
		\hat{\psi}_{FT}(\partial_{x} u + 1, \theta + T_{0}) \partial_{xt} u\big) = \hat{q}'(\partial_{x} \theta) \partial_{xx} \theta &\text{ in } (0, 1) \times (0, \infty)
		\label{SLEMROD_THERMOELASTICITY_PDE_2}
	\end{align}
	for the unknown functions $u$ denoting the displacement of the rod 
	and $\theta$ being a temperature difference to a reference temperature $T_{0}$
	rendering the body free of thermal stresses.
	The functions $\hat{\psi}$ and $\hat{q}$ denote the Helmholtz free energy and the heat flux, respectively,
	and are assumed to be given. Finally, $\rho > 0$ is the material density in the references configuration.
	Under appropriate boundary conditions
	(when the boundary is free of tractions and is held at a constant temperature or
	when the body is rigidly clamped and thermally insulated)
	as well as usual initial conditions for both unknown functions,
	a local existence theorem for Equations (\ref{SLEMROD_THERMOELASTICITY_PDE_1})--(\ref{SLEMROD_THERMOELASTICITY_PDE_2}) was proved
	by additionally imposing a regularity and compatibility condition.
	For sufficiently small initial data, the local classical solution could be globally continued.
	At the same time, when studying Equations (\ref{SLEMROD_THERMOELASTICITY_PDE_1})--(\ref{SLEMROD_THERMOELASTICITY_PDE_2}) in the whole space, 
	large data are known to lead to a blow-up in final time (cf. \cite{DaHsi1986}).

	Racke and Shibata studied in \cite{RaShi1991} Equations (\ref{SLEMROD_THERMOELASTICITY_PDE_1})--(\ref{SLEMROD_THERMOELASTICITY_PDE_2})
	under homogeneous Dirichlet boundary conditions for both $u$ and $\theta$.
	Under appropriate smoothness assumptions, they proved the global existence and exponential stability
	for the classical solutions to the problem.
	In contrast to Slemdrod \cite{Sle1981}, their method was using spectral analysis
	rather then ad hoc energy estimates obtained by differentiating the equations with respect to $t$ and $x$.
	A detailed overview of further recent developments in the field of classical thermoelasticity
	and corresponding references can be found in the monograph \cite{JiaRa2000} by Jiang and Racke.

	The classical equations of thermoelasticity outlined above, being a hyperbolic-parabolic system,
	provide a rather good macroscopic description in many real-world applications.
	At the same time, they sometimes fail when being used to model thermoelastic stresses in some other situations,
	in particular, in extremely small bodies exposed to heat pulses of large amplitude (see, e.g., \cite{WaXu2001}), etc.
	To address these issues, a new theory, commonly referred to as the theory of hyperbolic thermoelasticity
	or second sound thermoelasticity, has emerged.
	In contrast to the classical thermoelasticity,
	parabolic Equation (\ref{DAFERMOS_THERMOELASTICITY_PDE_2}) is replaced with a hyperbolic first-order system
	\begin{align}
		\rho c_{D} \partial_{t} T + m_{ij} \theta_{0} \partial_{t} u_{i, j} = q_{i, i} + \rho c_{D} r \text{ in } \Omega \times (0, \infty)
		\label{LINEAR_HYPERBOLIC_THERMOELASCITIY_PDE_1} \\
		\tau_{ij} \partial_{t} q_{i} + q_{i} + K_{ij} T_{, j} = 0 \text{ in } \Omega \times (0, \infty)
		\label{LINEAR_HYPERBOLIC_THERMOELASCITIY_PDE_2}
	\end{align}
	with $[q_{i}]$ and $[\tau_{ij}]$ denoting the heat flux and the relaxation tensor, respectively.
	Both linear and nonlinear versions of the equations of hyperbolic thermoelasticity
	(\ref{DAFERMOS_THERMOELASTICITY_PDE_1}), (\ref{LINEAR_HYPERBOLIC_THERMOELASCITIY_PDE_1})--(\ref{LINEAR_HYPERBOLIC_THERMOELASCITIY_PDE_2})
	have been studied in the literature.
	See, e.g., the article \cite{MeSa2005} by Messaoudi and Said-Houari for a proof of global well-posedness of the 1D system in the whole space
	or Irmscher's work \cite{Ir2011} for the global well-posedness of nonlinear problem
	for rotationally symmetric data in a bounded rotationally symmetric domain of $\mathbb{R}^{3}$.
	In a bounded 1D domain, a quantitative stability comparison
	between the classical and the hyperbolic system was presented by Irmscher and Racke in \cite{IrRa2006}.
	For a detailed overview on hyperbolic thermoelasticity,
	we refer the reader to the paper \cite{Cha1998} by Chandrasekharaiah
	and the work \cite{Ra2009} by Racke.

	A unified approach establishing a connection between the classical and hyperbolic thermoelasticity was established by Tzou in \cite{Tzou1995, Tzou1997}.
	Namely, he proposed to view Equation (\ref{LINEAR_HYPERBOLIC_THERMOELASCITIY_PDE_2}) with $\tau_{ij} \equiv \tau$ as a first-order Taylor approximation of the equation
	\begin{equation}
		q_{i}(\mathbf{x}, t + \tau) + K_{ij} T_{, j}(\mathbf{x}, t) = 0 \text{ for } (\mathbf{x}, t) \in \Omega \times (-\tau, \infty) \notag
	\end{equation}
	being equivalent to the delay equation
	\begin{equation}
		q_{i}(\mathbf{x}, t) + K_{ij} T_{, j}(\mathbf{x}, t - \tau) = 0 \text{ for } (\mathbf{x}, t) \in \Omega \times (0, \infty).
		\label{DELAY_EQUATION_TZOU}
	\end{equation}
	More generaly, a higher-order Taylor expansion to the dual-phase lag constitutive equation
	\begin{equation}
		q_{i}(\mathbf{x}, t + \tau_{1}) + K_{ij} T_{, j}(\mathbf{x}, t + \tau_{2}) = 0 \text{ for } (\mathbf{x}, t) \in \Omega \times (-\max\{\tau_{1}, \tau_{2}\}, \infty). \notag
	\end{equation}
	Together with Equations (\ref{DAFERMOS_THERMOELASTICITY_PDE_1})--(\ref{DAFERMOS_THERMOELASTICITY_PDE_2}), (\ref{LINEAR_HYPERBOLIC_THERMOELASCITIY_PDE_1}),
	this lead to the so-called dual phase-lag thermoelasticity studied by Quintanilla and Racke (cf. references on \cite[p. 415]{Ra2009}).

	If no Taylor expansion with respect to $\tau$ is carried out in Equation (\ref{DELAY_EQUATION_TZOU}),
	there can be shown that the corresponding system is ill-posed
	when being considered in the same topology as the original system of classical thermoelasticity (cf. \cite{DreQuiRa2009}),
	i.e., the system is lacking a continuous dependence of solution on the data.
	Moreover, the delay law (\ref{DELAY_EQUATION_TZOU})
	can, in general, contradict the second law of thermodynamics as shown in \cite{FaFra2014}.

	Nonetheless, it remains desirable to understand the dynamics of equations of thermoelasticity 
	orgininated from delayed material laws.
	One of the first attempt to obtain a well-posedness result for a partial differential equation with pure delay is due to Rodrigues et al.
	In their paper \cite{Ro2007}, Rodrigues et al. studied a heat equation with pure delay
	in an appropriate Frech\'{e}t space
	and showed the delayed Laplacian to generate a $C_{0}$-semigroup on this space.
	Further, they investigated the spectrum of the infinitesimal generator.
	Though their approach can essentially be carried over to the equations of thermoelasticity with pure delayed derived in Section \ref{SECTION_MODEL_DESCRIPTION} below,
	we propose a new approach in this paper preserving the Hilbert space structure of the space
	and thus the connection to the classical equations of thermoelasticity.
	To the authors' best knowledge,
	no results on thermoelasticity with delay in the highest order terms have been previously published in the literature.
	At the same time, we refer the reader 
	to the works by Khusainov et al. \cite{KhuIvKo2009, KhuPoAz2013.1, KhuPoAz2013.2, KhuPoAz2014},
	in which the authors studied the well-posedness and controllability
	for the heat and/or the wave equation on a finite time horizon.
	In their recent paper \cite{KhuPoRa2013}, Khusainov et al.
	exploited the $L^{2}$-maximum regularity theory
	to prove a global well-posedness and asymptotic stability results for a regularized heat equation.

	The present article has the following outline.
	In Section \ref{SECTION_MODEL_DESCRIPTION},
	we give a physical model for linear thermoelasticity based on delayed material laws.
	For the sake of simplicity, we present a 1D model
	though our approach can easily be carried over to the general multidimensional case.
	Next, in Section \ref{SECTION_WELL_POSEDNESS},
	we prove the well-posedness of this model in an appropriate Hilber space framework
	and discuss the small parameter asymptotics, i.e., the behavior of solutions as $\tau \to 0$.
	Further, in Section \ref{SECTION_EXPLICIT_SOLUTION_REPRESENTATION},
	we deduce an explicit solution representation formula.
	Finally, in the Appendix, we summarize some seminal results on the delayed exponential function
	and Cauchy problems with pure delay.

\section[Model Description]{Model Description}
	\label{SECTION_MODEL_DESCRIPTION}
	We consider a solid body occupying an axis aligned rectangular domain of $\mathbb{R}^{3}$.
	Assuming that the body motion is purely longitudinal with respect to the first space variable $x$ (cf. \cite[p. 100]{Sle1981}),
	deformation gradient, stress and strain tensors, etc., are diagonal matrices
	and a complete rational description of the original 3D body motion can be reduced to studying
	the 1D projection $\Omega = (0, l)$, $l > 0$, of the body onto the $x$-axis
	as displayed on Figure \ref{FIGURE_SOLID_BODY} below.
	Hence, in the following, we restrict ourselves to considering the relevant physical values only in $x$-direction.

	Let the functions $u \colon \bar{\Omega} \times [0, \infty) \to \mathbb{R}$ and $\theta \colon \bar{\Omega} \times [0, \infty) \to \mathbb{R}$
	denote the body displacement and its relative temperature measured 
	with respect to a reference temperature $\theta_{0} > 0$ rendering the body free of thermal stresses, respectively.
	We restrict ourselves to the Lagrangian coordinates and
	write $\sigma, \varepsilon, S, q \colon \bar{\Omega} \times [0, \infty) \to \mathbb{R}$ for the stress field, strain field, entropy field 
	or the heat flux, respectively.
	\begin{figure}[h!]
		\begin{centering}
			\setlength{\unitlength}{1mm}
			\begin{picture}(160, 65)(0,0)
				\put(78, 58){$x$}
				\put(125, 23){$y$}
				\put(99, 39){$z$}
				\put(81, 31){$l$}
				\put(20, 0){\includegraphics[scale = 0.4]{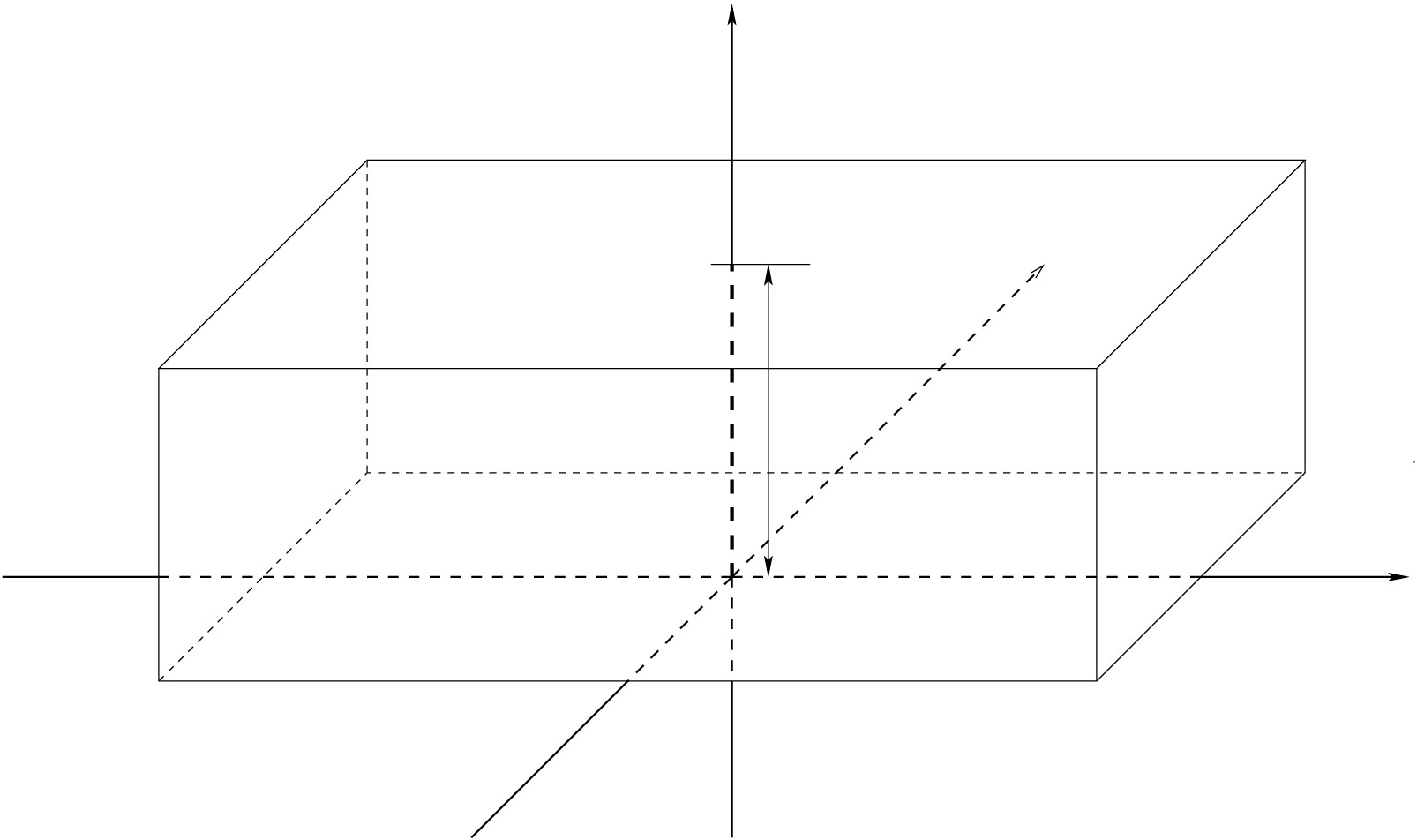}}
			\end{picture}
		\end{centering}
		\caption{3D rectangular solid body \label{FIGURE_SOLID_BODY}}
	\end{figure}
	With $\rho > 0$ denoting the material density, the momentum conservation law as well as the linearized entropy balance law read as
	\begin{align}
		\rho \partial_{tt} u(x, t) + \partial_{x} \sigma(x, t) &= \rho r(x, t) \text{ for } x \in \Omega, t > 0, \label{EQUATION_CONSERVATION_LAW_1} \\
		\theta_{0} \partial_{t} S(x, t) + \partial_{x} q(x, t) &= h(x, t) \text{ for } x \in \Omega, t > 0 \label{EQUATION_CONSERVATION_LAW_2}
	\end{align}
	where $r \colon \bar{\Omega} \times [0, \infty) \to \mathbb{R}$ and $h \colon \bar{\Omega} \times [0, \infty) \to \mathbb{R}$ are a known external force
	acting on the body and a heat source.
	
	Assuming physical linearity for the strain field, 
	the strain can be decomposed into elastic strain $\varepsilon^{e}$ and thermal stress $\varepsilon^{t}$.
	Furthing, assuming $\left|\frac{\theta(t, x)}{\theta_{0}}\right| \ll 1$ uniformly with respect to $x \in \bar{\Omega}$, $t \geq 0$,
	we can postulate
	\begin{equation}
		\varepsilon^{t}(x, t) = \alpha \theta(x, t) \text{ for } x \in \Omega, t > 0 \notag
	\end{equation}
	where $\alpha > 0$ denotes the thermal expansion coefficient.
	Exploiting the second law of thermodynamics for irreversible processes, we obtain (cf. \cite[p. 3]{Day1985})
	\begin{equation}
		S(x, t) = \alpha B \varepsilon^{e}(x, t) + \tfrac{\rho c_{\rho}}{\theta_{0}} \theta(x, t) \text{ for } x \in \Omega, t > 0
		\label{EQUATION_ENTROPY}
	\end{equation}
	with $c_{\rho} > 0$ standing for the specific heat capacity and $B \in \mathbb{R}$ denoting the bulk modulus.

	In our further considerations, we depart from the classical material laws and use their delay counterparts.
	Let $\tau > 0$ be a positive time delay.
	In the sequel, all functions are supposed to be defined on $\bar{\Omega} \times [-\tau, \infty)$.
	Assuming a delay feedback between the stress and the strain as well as the heat flux and the temperature gradient,
	the Hooke's law with pure delay reads as (cp. \cite{Ca1972})
	\begin{equation}
		\sigma(x, t) = \big(B + \tfrac{4}{3} G\big) \varepsilon^{e}(x, t - \tau) + B \varepsilon^{t}(x, t - \tau) \text{ for } x \in \Omega, t > 0
		\label{EQUATION_HOOKES_LAW}
	\end{equation}
	with $G > 0$ denoting the shear modulus.
	Similarly, we consider a delay version of Fourier's law given as
	\begin{equation}
		q(x, t) = -\kappa \partial_{x} \theta (x, t - \tau) \text{ for } x \in \Omega, t > 0
		\label{EQUATION_FOURIER_LAW}
	\end{equation}
	where $\kappa > 0$ stands for the thermal conductivity.
	Assuming the elastic strain tensor to be equal to the displacement gradient, we have
	\begin{equation}
		\varepsilon^{e}(x, t) = \partial_{x} u(x, t) \text{ for } x \in \Omega, t > 0.
		\label{EQUATION_STRAIN}
	\end{equation}
	Since within the infinitesimal elasticity theory
	the stress tensor $\sigma(x, t)$ and the deformation $\partial_{x} u(x, t)$ must be proportional,
	Equations (\ref{EQUATION_HOOKES_LAW}) and (\ref{EQUATION_STRAIN}) imply together
	\begin{equation}
		\partial_{x} u(x, t) = \partial_{x} u(x, t - \tau) \text{ for } x \in \Omega, t > 0.
		\label{EQUATION_DELAY_SCHWARZ_THEOREM}
	\end{equation}
	Finally, we also modify (\ref{EQUATION_ENTROPY}) to introduce a delay feedback between the entropy,
	the elastic strain tensor and the temperature
	\begin{equation}
		S(x, t) = \alpha B \varepsilon^{e}(x, t - \tau) + \tfrac{\rho c_{\rho}}{\theta_{0}} \theta(x, t - \tau) \text{ for } x \in \Omega, t > 0.
		\label{EQUATION_ENTROPY_DELAY}
	\end{equation}
	Exploiting now Equations (\ref{EQUATION_CONSERVATION_LAW_1}), (\ref{EQUATION_CONSERVATION_LAW_2}), (\ref{EQUATION_HOOKES_LAW})--(\ref{EQUATION_ENTROPY_DELAY}), we obtain
	\begin{align}
		\rho \partial_{tt} u(x, t) - \big(B + \tfrac{4}{3}G\big) \partial_{xx} u(x, t - \tau) + \alpha B \partial_{x} \theta(x, t - \tau)  &= f(x, t) \text{ for } x \in \Omega, t > 0, \label{EQUATION_THERMOELASTICITY_DELAY_1} \\
		\rho c_{\rho} \partial_{t} \theta(x, t) - \kappa \partial_{xx} \theta(x, t - \tau) + \alpha \theta_{0} B \partial_{tx} u(x, t - \tau) &= h(x, t) \text{ for } x \in \Omega, t > 0, \label{EQUATION_THERMOELASTICITY_DELAY_2} \\
		\partial_{t} \partial_{x} u(x, t) - \partial_{x} \partial_{t} u(x, t - \tau) &= 0 \text{ for } x \in \Omega, t > 0. \label{EQUATION_THERMOELASTICITY_DELAY_3}
	\end{align}
	
	To close Equations (\ref{EQUATION_THERMOELASTICITY_DELAY_1})--(\ref{EQUATION_THERMOELASTICITY_DELAY_3}),
	appropriate boundary and initial conditions for $u$ and $\theta$ are required.
	In the following, we prescribe homogeneous Dirichlet boundary conditions for $u$ and homogeneous Neumann boundary conditions for $\theta$ given as
	\begin{equation}
		u(0, t) = u(l, t) = 0, \quad \partial_{x} \theta(0, t) = \partial_{x} \theta(l, t) = 0 \text{ for } t > 0.
		\label{EQUATION_BOUNDARY_CONDITIONS}
	\end{equation}
	This particular choice of boundary conditions not only turns out to be convenient for our further mathematical considerations
	but is also a physically relevant one.
	Similar to the thermoelasticity with second sound,
	it is one of the combinations typically arising when studying micro- and nanoscopic strings or plates (cp. \cite{IrRa2006}).

	The initial conditions are given over the whole history period $(\tau, 0)$ and read as
	\begin{equation}
		\begin{array}{ccl}
			\phantom{\partial_{t}} u(x, 0) = u^{0}(x), & \phantom{\partial_{t}} u(x, t) = u^{0}(x, t) & \text{ for } x \in \Omega, t \in (-\tau, 0), \\
			\partial_{t} u(x, 0) = u^{1}(x), & \partial_{t} u(x, t) = u^{1}(x, t) & \text{ for } x \in \Omega, t \in (-\tau, 0), \\
			\phantom{\partial_{t}} \theta(x, 0) = \theta^{0}(x), & \phantom{\partial_{t}} \theta(x, t) = \theta^{0}(x, t) & \text{ for } x \in \Omega, t \in (-\tau, 0)
		\end{array}
		\label{EQUATION_INITIAL_CONDITIONS}
	\end{equation}
	with known
	$u^{0}, u^{1}, \theta^{0} \colon \Omega \to \mathbb{R}$ and
	$u^{0}_{\tau}, u^{1}_{\tau}, \theta^{0}_{\tau} \colon \Omega \times (-\tau, 0) \to \mathbb{R}$.

\section[Well-Posedness and Limit tau to 0]{Well-Posedness and Limit $\tau \to 0$}
	\label{SECTION_WELL_POSEDNESS}
	Letting $a := \tfrac{B + \tfrac{4}{3} G}{\rho}$, $b := \tfrac{\alpha B}{\rho}$,
	$c := \tfrac{\kappa}{\rho c_{\rho}}$, $d := \tfrac{\alpha \theta_{0} B}{\rho c_{\rho}}$ and
	$f(x, t) := r(x, t)$, $g(x, t) := \tfrac{1}{\rho c_{\rho}} h(x, t)$ for $x \in \bar{\Omega}$, $t \geq 0$,
	Equations (\ref{EQUATION_THERMOELASTICITY_DELAY_1})--(\ref{EQUATION_THERMOELASTICITY_DELAY_3}) can be re-written as
	\begin{align}
		\partial_{tt} u(x, t) - a \partial_{xx} u(x, t - \tau) + b \partial_{x} \theta(x, t - \tau)  &= f(x, t) \text{ for } x \in \Omega, t > 0, \label{EQUATION_THERMOELASTICITY_DELAY_REWRITTEN_1} \\
		\partial_{t} \theta(x, t) - c \partial_{xx} \theta(x, t - \tau) + d \partial_{tx} u(x, t - \tau) &= g(x, t) \text{ for } x \in \Omega, t > 0, \label{EQUATION_THERMOELASTICITY_DELAY_REWRITTEN_2} \\
		\partial_{x} u(x, t) - \partial_{x} u(x, t - \tau) &= 0 \text{ for } x \in \Omega, t > 0 \label{EQUATION_THERMOELASTICITY_DELAY_REWRITTEN_3}
	\end{align}
	subject to the boundary conditions from Equation (\ref{EQUATION_BOUNDARY_CONDITIONS})
	and initial conditions from Equation (\ref{EQUATION_INITIAL_CONDITIONS}).
	Introducing a new vector of unknown functions
	\begin{equation}
		\mathbf{V}(x, t) =
		\begin{pmatrix}
			V^{1}(x, t) \\
			V^{2}(x, t) \\
			V^{3}(x, t)
		\end{pmatrix}
		:=
		\begin{pmatrix}
			\partial_{t} u(x, t) \\
			\partial_{x} u(x, t) \\
			\theta(x, t)
		\end{pmatrix}
		\text{ for } x \in \bar{\Omega}, t \in [-\tau, T], 
		\notag
	\end{equation}
	Equations (\ref{EQUATION_THERMOELASTICITY_DELAY_REWRITTEN_1})--(\ref{EQUATION_THERMOELASTICITY_DELAY_REWRITTEN_3}) can be transformed to
	\begin{equation}
		\partial_{t} \mathbf{V}(x, t) + \mathbf{B} \mathbf{V}(x, t - \tau) = \mathbf{F}(x, t) \text{ for } x \in \Omega, t \in (0, T)
		\label{EQUATION_FIRST_ORDER_SYSTEM_PDE}
	\end{equation}
	with the differential matrix operator and the right-hand side
	\begin{equation}
		\mathbf{B} :=
		\begin{pmatrix}
			0 & -a \partial_{x} & b \partial_{x} \\
			-\partial_{x} & 0 & 0 \\
			d \partial_{x} & 0 & -c \partial_{xx}
		\end{pmatrix} \text{ and }
		\mathbf{F}(x, t) :=
		\begin{pmatrix}
			f(x, t) \\
			0 \\
			g(x, t)
		\end{pmatrix}
		\text{ for } x \in \bar{\Omega}, t > 0, \text{ respectively.}
        	\notag
	\end{equation}
	Exploiting Equation (\ref{EQUATION_BOUNDARY_CONDITIONS}) and the definition of $V$, the boundary conditions for $V$ read as
	\begin{equation}
		V^{1}(0, t) = V^{1}(l, t) = 0, \quad
		\partial_{x} V^{3}(0, t) = \partial_{x} V^{3}(l, t) = 0 \text{ for } t > 0
		\label{EQUATION_FIRST_ORDER_SYSTEM_BOUNDARY_CONDITIONS}
	\end{equation}
	whereas the initial conditions are given by
	\begin{equation}
		\mathbf{V}(x, 0) = \mathbf{V}^{0}(x), \quad
		\mathbf{V}(x, t) = \mathbf{V}^{0}_{\tau}(x, t) \text{ for } x \in \Omega, t \in (-\tau, 0)
		\label{EQUATION_FIRST_ORDER_SYSTEM_INITIAL_CONDITIONS}
	\end{equation}
	with
	\begin{equation}
		\mathbf{V}^{0}(x) =
		\begin{pmatrix}
			u^{0} \\ u^{1} \\ \theta^{0}
		\end{pmatrix}, \quad
		\mathbf{V}^{0}_{\tau}(x, t) =
		\begin{pmatrix}
			u^{1}_{\tau}(x, t) \\
			\partial_{x} u^{0}_{\tau}(x, t) \\
			\theta^{0}_{\tau}(x, t)
		\end{pmatrix}
		\text{ for } x \in \bar{\Omega}, t \in [-\tau, 0].
		\notag
	\end{equation}
	Note that Equations
	(\ref{EQUATION_THERMOELASTICITY_DELAY_1})--(\ref{EQUATION_INITIAL_CONDITIONS})
	and (\ref{EQUATION_FIRST_ORDER_SYSTEM_PDE})--(\ref{EQUATION_FIRST_ORDER_SYSTEM_INITIAL_CONDITIONS}) are equivalent
	for, if the vector $V$ is known, $u$ and $\theta$ are uniquely determined by
	\begin{equation}
		u(x, t) =
		\begin{cases}
			u^{0}(x) + \int_{0}^{t} V^{1}(x, s) \mathrm{d}s, & \text{ for } t \geq 0, \\
			u^{0}_{\tau}(x, t), & \text{ for } t \in [-\tau, 0),
		\end{cases} \quad
		\theta(t, x) =
		\begin{cases}
			V^{3}(x, t), & \text{ for } t \geq 0, \\
			\theta^{0}_{\tau}(x, t), & \text{ for } t \in [-\tau, 0).
		\end{cases}
		\notag
	\end{equation}
	Therefore, in the sequel, we consider the following equivalent first-order-in-time problem
	\begin{align}
		&\partial_{t} \mathbf{V}(x, t) + \mathbf{B} \mathbf{V}(x, t - \tau) = \mathbf{F}(x, t) \text{ for } x \in \Omega, t > 0, \label{EQUATION_FIRST_ORDER_SYSTEM_1} \\
		&V^{1}(0, t) = V^{1}(l, t) = 0, \quad \partial_{x} V^{3}(0, t) = \partial_{x} V^{3}(l, t) = 0 \text{ for } t > 0, \label{EQUATION_FIRST_ORDER_SYSTEM_2} \\
		&\mathbf{V}(x, 0) = \mathbf{V}^{0}(x), \quad \mathbf{V}(x, t) = \mathbf{V}^{0}_{\tau}(x, t) \text{ for } x \in \Omega, t \in (-\tau, 0). \label{EQUATION_FIRST_ORDER_SYSTEM_3}
	\end{align}
	
	For our well-posedness investigations,
	we need a solution notion for Equations (\ref{EQUATION_FIRST_ORDER_SYSTEM_1})--(\ref{EQUATION_FIRST_ORDER_SYSTEM_3}).
	To this end, appropriate functional spaces have to be introduced.
	We start with the ``naive'' approach by using the case $\tau = 0$ as a reference situation.
	We introduce the Hilbert space $X := L^{2}(\Omega) \times L^{2}(\Omega) \times L^{2}(\Omega)$ equipped with the dot product
	\begin{equation}
		\langle \mathbf{V}, \mathbf{W}\rangle_{X} :=
		\langle V^{1}, W^{1}\rangle_{L^{2}(\Omega)} + a \langle V^{2}, W^{2}\rangle_{L^{2}(\Omega)} + \tfrac{b}{d} \langle V^{3}, W^{3}\rangle_{L^{2}(\Omega)}
		\text{ for } \mathbf{V}, \mathbf{W} \in X \notag
	\end{equation}
	and define the operator
	\begin{equation}
		\mathcal{B} \colon D(\mathcal{B}) \subset X \to X, \quad V \mapsto BV \notag
	\end{equation}
	with the domain
	\begin{equation}
		D(\mathcal{B}) := \big\{\mathbf{V} \in H^{1}_{0}(\Omega) \times H^{1}(\Omega) \times H^{2}(\Omega) \,|\,
		\partial_{x} V^{3}|_{\partial \Omega} = 0\big\}. \notag
	\end{equation}
	See \cite[Section 3]{Ad2003} for the definition of Sobolev spaces.
	With this notation, Equations (\ref{EQUATION_FIRST_ORDER_SYSTEM_1})--(\ref{EQUATION_FIRST_ORDER_SYSTEM_3}) can be written in the equivalent form
	\begin{align}
		&\partial_{t} \mathbf{V}(x, t) + \mathcal{B} \mathbf{V}(x, t) = \mathbf{F}(x, t) \text{ for } x \in \Omega, t > 0, \label{EQUATION_FIRST_ORDER_SYSTEM_ABSTRACT_FORM_B_1} \\
		&\mathbf{V}(x, 0) = \mathbf{V}^{0}(x), \quad \mathbf{V}(x, t) = \mathbf{V}^{0}_{\tau}(x, t) \text{ for } x \in \Omega, t \in (-\tau, 0). \label{EQUATION_FIRST_ORDER_SYSTEM_ABSTRACT_FORM_B_2}
	\end{align}
	Under a classical solution to Equations (\ref{EQUATION_FIRST_ORDER_SYSTEM_ABSTRACT_FORM_B_1})--(\ref{EQUATION_FIRST_ORDER_SYSTEM_ABSTRACT_FORM_B_2}),
	one would naturally understand a function
	$\mathbf{V} \in C^{0}\big([-\tau, \infty), X\big) \cap C^{1}\big([0, \infty), D(\mathcal{B})\big)$ satisfying the equations pointwise.

	We know from \cite{Ha1992}
	that the linear operator $\mathcal{B}$ is skew-selfadjoint and accretive.
	Its spectrum $\sigma(\mathcal{B})$ only consists of isolated eigenvalues
	$\lambda_{n} \in \mathbb{C}$, $n \in \mathbb{N}_{0}$, of finite multiplicity
	with $\mathrm{Re} \, \lambda_{n} \geq 0$, $n \in \mathbb{N}$, and $\lambda_{n} \to \infty$ as $n \to \infty$.
	The corresponding eigenfunctions $(\boldsymbol{\Psi}_{n})_{n} \subset D(\mathcal{B})$
	build an orthonormal basis of $X$.
	Unfortunately, from \cite[Theorem 1.1]{DreQuiRa2009} we know that Equations (\ref{EQUATION_FIRST_ORDER_SYSTEM_ABSTRACT_FORM_B_1})--(\ref{EQUATION_FIRST_ORDER_SYSTEM_ABSTRACT_FORM_B_2}) are ill-posed in $X$.
	Hence, a different solution notion should be adopted.
	As we already mentioned in Introduction,
	we want to preserve the Hilbert space structure of the problem
	and thus cannot follow the approach developed by Rodrigues et al. in \cite{Ro2007}.
	
	We define the space $X_{\infty} := \Big\{V \in \bigcap\limits_{n = 0}^{\infty} D(\mathcal{B}^{n}) \,\big|\, \|V\|_{X_{\infty}} < \infty\Big\}$
	equipped with the scalar product
	\begin{equation}
		\langle \mathbf{V}, \mathbf{W}\rangle_{X_{\infty}} :=
		\sum_{k = 0}^{\infty} \tfrac{1}{k!} \langle \mathcal{B} \mathbf{V}, \mathcal{B} \mathbf{W}\rangle_{X}
		\text{ for } \mathbf{V}, \mathbf{W} \in X. \notag
	\end{equation}
	Obviously, $X_{\infty}$ is a Hilbert space.
	Moreover, $X_{\infty}$ is dense in $X$ since $(\mathbf{\Phi}_{n})_{n} \subset X_{\infty}$.
	Indeed, for $n \in \mathbb{N}$, we have
	\begin{equation}
		\|\mathbf{\Psi}_{n}\|_{X_{\infty}}^{2} =
		\sum_{k = 0}^{\infty} \tfrac{1}{k!} |\lambda_{n}|^{2k} \|\mathbf{\Psi}_{n}\|_{X}^{2} =
		\sum_{k = 0}^{\infty} \tfrac{1}{k!} |\lambda_{n}|^{2k} = \exp(\lambda_{n}^{2}) < \infty. \notag
	\end{equation}
	
	Restricting $\mathcal{B}$ to its closed subspace $X_{\infty}$,
	we obtain a bounded linear operator
	$\mathcal{B}_{\infty} := \mathcal{B}|_{X_{\infty}}$ on $X_{\infty}$ since
	\begin{equation}
		\|\mathcal{B}_{\infty} \mathbf{V}\|_{X_{\infty}}^{2} =
		\sum_{k = 0}^{\infty} \tfrac{1}{k!} \|\mathcal{B}^{k} \mathcal{B} \mathbf{V}\|_{X}^{2} =
		\sum_{k = 1}^{\infty} \tfrac{1}{k!} \|\mathcal{B}^{k+1} \mathbf{V}\|_{X}^{2} \leq
		\|\mathbf{V}\|_{X_{\infty}}^{2}
		\text{ for any } \mathbf{V} \in X_{\infty}. \notag
	\end{equation}
	Now, restricting Equations (\ref{EQUATION_FIRST_ORDER_SYSTEM_ABSTRACT_FORM_B_1})--(\ref{EQUATION_FIRST_ORDER_SYSTEM_ABSTRACT_FORM_B_2}) to $X_{\infty}$, we obtain
	\begin{align}
		&\partial_{t} \mathbf{V}(x, t) + \mathcal{B}_{\infty} \mathbf{V}(x, t - \tau) = \mathbf{F}(x, t) \text{ for } x \in \Omega, t > 0, \label{EQUATION_FIRST_ORDER_SYSTEM_ABSTRACT_FORM_B_INFTY_1} \\
		&\mathbf{V}(x, 0) = \mathbf{V}^{0}(x), \quad \mathbf{V}(x, t) = \mathbf{V}^{0}_{\tau}(x, t) \text{ for } x \in \Omega, t \in (-\tau, 0). \label{EQUATION_FIRST_ORDER_SYSTEM_ABSTRACT_FORM_B_INFTY_2}
	\end{align}

	Applying Theorem \ref{THEOREM_DELAY_ODE} from Appendix,
	we get the following well-posedness result.
	\begin{theorem}
		\label{THEOREM_EXISTENCE_OF_CLASSICAL_SOLUTIONS}
		Let $\mathbf{V}^{0} \in X_{\infty}$, $\mathbf{V}^{0}_{\tau} \in C^{0}\big([-\tau, 0], X_{\infty}\big)$ with $\mathbf{V}^{0}_{\tau}(\cdot, 0) = \mathbf{V}^{0}$
		and let $\mathbf{F} \in C^{0}\big([0, \infty), X_{\infty}\big)$.
		Then Equations
		(\ref{EQUATION_FIRST_ORDER_SYSTEM_ABSTRACT_FORM_B_INFTY_1})--(\ref{EQUATION_FIRST_ORDER_SYSTEM_ABSTRACT_FORM_B_INFTY_2})
		possess a unique classical solution
		$\mathbf{V} \in C^{0}\big([-\tau, \infty), X_{\infty}\big) \cap C^{1}\big([0, \infty), X_{\infty}\big)$
		explicitly given as
		\begin{equation}
			\mathbf{V}(\cdot, t) = \left\{
			\begin{array}{cl}
				\mathbf{V}^{0}_{\tau}(\cdot, t), & t \in [-\tau, 0), \\
				\mathbf{V}^{0}, & t = 0, \\
				{\exp_{\tau}(-\mathcal{B}_{\infty}, t - \tau) \mathbf{V}^{0} -
				\mathcal{B}_{\infty} \int_{-\tau}^{0} \exp_{\tau}(-\mathcal{B}_{\infty}, t - 2\tau - s) \mathbf{V}^{0}_{\tau}(s) \mathrm{d}s + \atop
				\int_{0}^{t} \exp_{\tau}(-\mathcal{B}_{\infty}, t - \tau - s) \mathbf{F}(\cdot, s) \mathrm{d}s}, & t \in (0, T]
			\end{array}\right.
			\notag
		\end{equation}
	\end{theorem}
	Taking into account the trivial estimate $\|\exp_{\tau}(-\mathcal{B}_{\infty}, t)\|_{X_{\infty}} \leq \exp(\|\mathcal{B}_{\infty}\|_{L(X_{\infty})} t) \leq \max\{1, \exp(t)\}$ for $t \in \mathbb{R}$
	and applying H\"older's inequality, we use the solution representation formula from Theorem \ref{THEOREM_EXISTENCE_OF_CLASSICAL_SOLUTIONS} to obtain the following estimate.
	\begin{corollary}
		The solution $\mathbf{V}$ continuously depends on the data in sense of the estimate
		\begin{equation}
			\|\mathbf{V}\|_{C^{0}([0, T], X_{\infty})} \leq
			\exp(T) \|\mathbf{V}^{0}\|_{X_{\infty}} +
			\tau \exp(T) \|\mathbf{V}^{0}_{\tau}\|_{C^{0}([0, T], X_{\infty})} +
			\sqrt{T} \exp(T) \|\mathbf{F}\|_{L^{2}(0, T; X_{\infty})} \text{ as } T > 0. \notag
		\end{equation}
	\end{corollary}
	
	For the rest of this section,
	we want to study the behavior of system (\ref{EQUATION_FIRST_ORDER_SYSTEM_ABSTRACT_FORM_B_1})--(\ref{EQUATION_FIRST_ORDER_SYSTEM_ABSTRACT_FORM_B_2})
	for $\tau \to 0$.
	Formally, the limitting system is given as
	\begin{align}
		&\partial_{t} \bar{\mathbf{V}}(x, t) + \mathcal{B}_{\infty} \bar{\mathbf{V}}(x, t) = \mathbf{F}(x, t) \text{ for } x \in \Omega, t > 0, \label{EQUATION_FIRST_ORDER_SYSTEM_ABSTRACT_FORM_B_INFTY_LIMIT_1} \\
		&\bar{\mathbf{V}}(x, 0) = \mathbf{V}^{0}(x) \text{ for } x \in \Omega. \label{EQUATION_FIRST_ORDER_SYSTEM_ABSTRACT_FORM_B_INFTY_LIMIT_2}
	\end{align}
	Being a bounded operator itself,
	$-\mathcal{B}_{\infty}$ generates an analytic $C_{0}$-semigroup of bounded linear operators
	$\exp(-\mathcal{B}_{\infty} t) = \sum\limits_{k = 0}^{\infty} \tfrac{(-\mathcal{B}_{\infty} t)^{k}}{k!}$ on $X_{\infty}$.
	The unique classical solution to Equations (\ref{EQUATION_FIRST_ORDER_SYSTEM_ABSTRACT_FORM_B_1})--(\ref{EQUATION_FIRST_ORDER_SYSTEM_ABSTRACT_FORM_B_2})
	can then be written using the Duhamel's formula
	\begin{equation}
		\bar{\mathbf{V}}(\cdot, t) =
		\exp(-\mathcal{B}_{\infty} t) \mathbf{V}^{0} +
		\int_{0}^{t} \exp(-\mathcal{B}_{\infty} (t - s)) \mathbf{F}(s) \mathrm{d}s \text{ for } t \geq 0. \notag
	\end{equation}

	\begin{lemma}
		For any $T > 0$ and $\tau > 0$, there holds
		\begin{equation}
			\|\exp_{\tau}(-\mathcal{B}_{\infty}, t - \tau) - \exp(-\mathcal{B}_{\infty} t)\|_{L(X_{\infty})} \leq \tau \exp(T) \text{ for } t \in [0, T]. \notag
		\end{equation}
	\end{lemma}
	\begin{proof}
		For $t \in [0, T]$, the claim is an obvious consequence of the mean value theorem.
		Now, taking into account this fact, we use the induction to prove for any natural $k \in \mathbb{N}$
		\begin{equation}
			\|\exp_{\tau}(-\mathcal{B}_{\infty}, t - \tau) - \exp(-\mathcal{B}_{\infty} t)\|_{L(X_{\infty})} 
			\leq \tau \exp(k \tau) \text{ for } t \in ((k - 1) \tau, k \tau].
			\notag
		\end{equation}
		Assuming the claim is true for some $k \in \mathbb{N}$, we want to prove the same assertion for $k + 1$.
		Using the induction assumption and the fundamental theorem of calculus, we get for $t \in (k \tau, (k + 1) \tau]$
		\begin{align*}
			\|\exp_{\tau}&(-\mathcal{B}_{\infty}, t - \tau) - \exp(-\mathcal{B}_{\infty} t)\|_{L(X_{\infty})} \\
			&\leq \tau \exp(k \tau) +
			\int_{k\tau}^{(k + 1) \tau} \|\partial_{s} \exp_{\tau}(-\mathcal{B}_{\infty}, s - \tau) - \partial_{s} \exp(-\mathcal{B}_{\infty} s)\|_{L(X_{\infty})} \\
			&\leq
			\tau \exp(k \tau) + \|\mathcal{B}_{\infty}\|_{L(X_{\infty})}
			\int_{k\tau}^{(k + 1) \tau} \|\exp_{\tau}(-\mathcal{B}_{\infty}, s - 2 \tau) - \exp(-\mathcal{B}_{\infty} s)\|_{L(X_{\infty})} \\
			&\leq
			\tau \exp(k \tau) + \int_{k\tau}^{(k + 1) \tau} \|\exp_{\tau}(-\mathcal{B}_{\infty}, s - \tau) - \exp(-\mathcal{B}_{\infty} s)\|_{L(X_{\infty})} \\
			&+ \int_{k\tau}^{(k + 1) \tau} \|\exp_{\tau}(-\mathcal{B}_{\infty}, s - \tau) - \exp_{\tau}(-\mathcal{B}_{\infty}, s - 2\tau)\|_{L(X_{\infty})} \\
			&\leq \tau \exp(k \tau) + \tau^{2} \exp(k \tau) + \tfrac{\tau^{2}}{k!} \leq
			\exp(k \tau) \big(1 + \tau + \tfrac{\tau^{2}}{2}\big)
			\leq \tau \exp((k + 1) \tau)
		\end{align*}
		since
		$\|\exp_{\tau}(-\mathcal{B}_{\infty}, t) - \exp_{\tau}(-\mathcal{B}_{\infty}, t - \tau)\|_{L(X_{\infty})} \leq
		\tfrac{\tau^{k+1}}{(k + 1)!} \text{ for } t \in (k\tau, (k+1)\tau], k \in \mathbb{N}$,
		by definition of the delayed exponential function.
	\end{proof}

	\begin{theorem}
		Let $T > 0$ and let $\mathbf{V}^{0} \in X_{\infty}$, $\mathbf{F} \in C^{0}\big([0, \infty), X_{\infty}\big)$ be fixed.
		For $\tau > 0$, let $\mathbf{V}^{0}_{\tau} \in C^{0}\big([-\tau, 0], X_{\infty}\big)$ with
		$\mathbf{V}^{0}_{\tau}(0) = \mathbf{V}^{0}$ and
		$\limsup\limits^{}_{\tau \to 0} \|\mathbf{V}^{0}_{\tau}\|_{L^{1}(0, \tau; X_{\infty})} < \infty$.
		Denoting with $V(\cdot; \tau)$ the classical solution of
		(\ref{EQUATION_FIRST_ORDER_SYSTEM_ABSTRACT_FORM_B_INFTY_1})--(\ref{EQUATION_FIRST_ORDER_SYSTEM_ABSTRACT_FORM_B_INFTY_2})
		corresponding to the initial data $\mathbf{V}^{0}$, $\mathbf{V}^{0}_{\tau}$ and the right-hand side $\mathbf{F}$, we have
		\begin{equation}
			\|\mathbf{V}(\cdot; \tau) - \bar{\mathbf{V}}(\cdot; \tau)\|_{C^{0}([0, T], X_{\infty}} = O(\tau) \text{ as } \tau \to 0. \notag
		\end{equation}
	\end{theorem}

	\begin{proof}
		Using the representation formulas for $\mathbf{V}$ and $\bar{\mathbf{V}}$, 
		we can estimate for any $t \in [0, T]$
		\begin{align*}
			\|\mathbf{V}(\cdot; \tau) - \bar{\mathbf{V}}(\cdot; \tau)\|_{X_{\infty}} &\leq
			\big\|\exp_{\tau}(-\mathcal{B}_{\infty}, t - \tau) - \exp(-\mathcal{B}_{\infty} t)\big\|_{L(X_{\infty})} \|\mathbf{V}^{0}\|_{X_{\infty}} \\
			&+ \big\|\mathcal{B}_{\infty}\|_{L(X_{\infty})} \int_{-\tau}^{0} \|\exp_{\tau}(-\mathcal{B}_{\infty}, t - 2\tau - s)\|_{L(X_{\infty})} \|\mathbf{V}^{0}_{\tau}(s)\|_{X_{\infty}} \mathrm{d}s \\
			&+ \int_{0}^{t} \|\exp_{\tau}(-\mathcal{B}_{\infty}, t - \tau - s) - \exp(-\mathcal{B}_{\infty} (t - s))\|_{L(X_{\infty})} \|\mathbf{F}(\cdot, s)\|_{X_{\infty}} \mathrm{d}s \\
			&\leq \tau \exp(T) \|\mathbf{V}^{0}\|_{X_{\infty}} + 
			\tau (1 + \tau) \exp(T) \limsup_{\tau \to 0} \|\mathbf{V}^{0}_{\tau}\|_{L^{1}(-\tau, 0; X_{\infty})} \\
			&+ \tau T \exp(T) \|\mathbf{F}\|_{L^{\infty}(0, T; X_{\infty})} = O(\tau) \text{ as } \tau \to 0. \notag
		\end{align*}
		This finishes the proof.
	\end{proof}

\section[Explicit Solution Representation]{Explicit Solution Representation}
	\label{SECTION_EXPLICIT_SOLUTION_REPRESENTATION}
	In this section, we want to deduce an explicit representation of solutions to Equations (\ref{EQUATION_FIRST_ORDER_SYSTEM_ABSTRACT_FORM_B_INFTY_1})--(\ref{EQUATION_FIRST_ORDER_SYSTEM_ABSTRACT_FORM_B_INFTY_2})
	in the form of a Fourier series with respect to an orthogonal basis $(\boldsymbol{\Phi}_{n})_{n \in \mathbb{N}_{0}}$ of $X$ (and thus of $X_{\infty}$) given by
	\begin{equation}
		\boldsymbol{\Phi}_{n}(x) =
		\left\{
		\begin{array}{cl}
			\sqrt{\tfrac{1}{2l}} \big(0, 1, 1\big)^{T}, & \text{ if } n = 0, \\
			\sqrt{\tfrac{2}{3l}} \big(\sin(\nu_{n} x), \cos(\nu_{n} x), \cos(\nu_{n} x)\big)^{T}, &
			\text{ otherwise }
		\end{array}\right.
		\text{ for } x \in \bar{\Omega}, n \in \mathbb{N}_{0} \notag
	\end{equation}
	with
	\begin{equation}
		\nu_{n} := \tfrac{\pi n}{L} \text{ for } n \in \mathbb{N}_{0}. \notag
	\end{equation}
	Note that the sequence $(\boldsymbol{\Phi}_{n})_{n \in \mathbb{N}_{0}}$ does not coincide, in general, 
	with the eigenfunctions $(\boldsymbol{\Psi}_{n})_{n \in \mathbb{N}_{0}}$
	but, at the same time, $(\boldsymbol{\Phi}_{n})_{n \in \mathbb{N}_{0}} \subset D(\mathcal{B}_{\infty})$
	consistutes a basis of $D(\mathcal{B}_{\infty})$.
	To this end, we assume that the conditions of Theorem \ref{THEOREM_EXISTENCE_OF_CLASSICAL_SOLUTIONS} are satisfied
	which yields a unique classical solution
	$\mathbf{V} \in C^{0}\big([-\tau, \infty), X_{\infty}\big) \cap C^{1}\big([0, \infty), X_{\infty}\big)$.

	Denoting $\boldsymbol{\Phi}_{n} = (\Phi_{n}^{1}, \Phi_{n}^{2}, \Phi_{n}^{3})^{T}$ and
	computing the component-wise Fourier coefficients
	\begin{equation}
		\begin{split}
			V^{0, k}_{n} &= \langle V^{0, k}, \Phi_{n}^{k}\rangle_{L^{2}(\Omega)}, \\
			V^{0, k}_{\tau, n}(t) &= \langle V^{0, k}_{\tau}(\cdot, t), \Phi_{n}^{k}\rangle_{L^{2}(\Omega)} \text{ for } t \in [-\tau, 0], \\
			F_{n}^{k}(t) &= \langle F^{k}(\cdot, t), \Phi_{n}^{k}\rangle_{L^{2}(\Omega)} \text{ for } t \geq 0
		\end{split}
		\notag
	\end{equation}
	for $n \in \mathbb{N}_{0}$ and $k = 1, 2, 3$,
	we get the following Fourier expansions
	\begin{equation}
		\begin{split}
			\mathbf{V}^{0} &= \sum_{n = 0}^{\infty} \Big(V^{0, 1}_{n} \Phi_{n}^{1}, V^{0, 2}_{n} \Phi_{n}^{2}, V^{0, 3}_{n} \Phi_{n}^{3}\Big), \\
			\mathbf{V}_{\tau}^{0}(\cdot, t) &= \sum_{n = 0}^{\infty} \Big(V^{0, 1}_{\tau, n} \Phi_{n}^{1}, V^{0, 2}_{n} \Phi_{n}^{2}, V^{0, 3}_{n} \Phi_{n}^{3}\Big) \text{ for } t \in [-\tau, 0], \\
			\mathbf{F}(\cdot, t) &= \sum_{n = 0}^{\infty} \Big(F^{1}_{n} \Phi_{n}^{1}, F^{2}_{n} \Phi_{n}^{2}, F^{3}_{n} \Phi_{n}^{3}\Big) \text{ for } t \geq 0
		\end{split}
		\notag
	\end{equation}
	uniformly in $\bar{\Omega}$.
	Similarly, the solution $\mathbf{V}$ can be expanded into Fourier series
	\begin{equation}
		\mathbf{V}(\cdot, t) = \sum_{n = 0}^{\infty} \big(V_{n}^{1}(t) \Phi_{n}^{1}, V_{n}^{2}(t) \Phi_{n}^{1}, V_{n}^{3}(t) \Phi_{n}^{1}\big) \notag
	\end{equation}
	for some $V_{n, k} \in C^{0}\big([-\tau, \infty), \mathbb{R}\big) \cap C^{1}\big([0, \infty), \mathbb{R}\big)$, $n \in \mathbb{N}_{0}$, $k = 1, 2, 3$,
	to be determined later.
	Using this ansatz and letting
	\begin{equation}
		\mathbf{B}_{n} :=
		\begin{pmatrix}
			 0             & a \nu_{n} & -b \nu_{n} \\
			-\nu_{n}   & 0             &  0 \\
			 d \nu_{n} & 0             &  c \nu_{n}^{2}
		\end{pmatrix}, \notag
	\end{equation}
	we observe that Equations (\ref{EQUATION_FIRST_ORDER_SYSTEM_ABSTRACT_FORM_B_INFTY_1})--(\ref{EQUATION_FIRST_ORDER_SYSTEM_ABSTRACT_FORM_B_INFTY_2})
	decompose into a sequence of ordinary delay differential equations
	\begin{align}
		\dot{\mathbf{V}}_{n}(t) &= \mathbf{B}_{n} \mathbf{V}_{n}(t - \tau) + \mathbf{F}_{n}(t) \text{ for } t > 0, \label{EQUATION_SECTION_EXPLICIT_SOLUTION_N_COEFFICIENT_1} \\
		\mathbf{V}_{n}(0) &= \mathbf{V}_{n}^{0}, \quad
		\mathbf{V}_{n}(t) = \mathbf{V}_{\tau, n}^{0}(t) \text{ for } t \in (-\tau, 0). \label{EQUATION_SECTION_EXPLICIT_SOLUTION_N_COEFFICIENT_2}
	\end{align}
	
	By the virtue of Theorem \ref{THEOREM_DELAY_ODE},
	for any $n \in \mathbb{N}_{0}$, the unique solution to Equations (\ref{EQUATION_SECTION_EXPLICIT_SOLUTION_N_COEFFICIENT_1})--(\ref{EQUATION_SECTION_EXPLICIT_SOLUTION_N_COEFFICIENT_2}) is given by
	\begin{equation}
		\mathbf{V}_{n}(t) = \left\{
		\begin{array}{cl}
			\mathbf{V}^{0}_{\tau, n}(t), & t \in [-\tau, 0), \\
			\mathbf{V}^{0}_{n}, & t = 0, \\
			{\exp_{\tau}(-\mathbf{B}_{n}, t - \tau) \mathbf{V}^{0}_{n} -
			\mathbf{B}_{n} \int_{-\tau}^{0} \exp_{\tau}(-\mathbf{B}_{n}, t - 2\tau - s) \mathbf{V}^{0}_{\tau, n}(s) \mathrm{d}s + \atop
			\int_{0}^{t} \exp_{\tau}(-\mathbf{B}_{n}, t - \tau - s) \mathbf{F}_{n}(s) \mathrm{d}s}, & t \in (0, T]
		\end{array}\right.
		\label{EQUATION_SECTION_EXPLICIT_SOLUTION_N_COEFFICIENT_SOLUTION}
	\end{equation}
	To explicitly compute the function given in Equation (\ref{EQUATION_SECTION_EXPLICIT_SOLUTION_N_COEFFICIENT_SOLUTION}),
	we have to diagonalize the matrix $\mathbf{B}_{n}$.
	
	\begin{lemma}
		\label{LEMMA_EIGENVALUES_OF_A_N}
		Let
		\begin{equation}
			\Delta_{0} = c^{2} \nu_{n}^{4} - 3 (a + bd) \nu_{n}^{2}, \quad
			\Delta_{1} = -2 c^{3} \nu_{n}^{6} + 9 c (a + bd) \nu_{n}^{4} - 27 ac \nu_{n}^{4}, \quad
			C = \sqrt[3]{\tfrac{1}{2} \big(\Delta_{1} + \sqrt{\Delta_{1}^{2} - 4 \Delta_{0}^{3}}\big)} \notag
		\end{equation}
		where $\sqrt{\cdot}$ and $\sqrt[3]{\cdot}$
		stand for the main branch of complex square and cubic roots.
		The spectrum of $\mathbf{B}_{n}$ consists of three eigenvalues
		\begin{equation}
			\mu_{n, k} =
			\begin{cases}
				0, & n = 0, \\
				\tfrac{1}{3} \big(c \nu_{n}^{2} - C e^{2 i k \pi/3} - e^{-2 i k \pi/3} \tfrac{\Delta_{0}}{C}\big), &
				\text{otherwise}
			\end{cases}\notag
		\end{equation}
		for $k = 0, 1, 2$ with $i$ denoting the imaginary unit.
	\end{lemma}
	
	\begin{proof}
		For $n = 0$, we have $\nu_{n} = 0$ and therefore $\mathbf{B}_{n} = 0_{3 \times 3}$.
		Hence, $0$ is the only eigenvalue of $\mathbf{B}_{n}$ with an algebraic multiplicity of $3$.
		
		Now, let us assume $n > 1$.
		To compute the eigenvalues of $\mathbf{B}_{n}$, we consider the characteristic polynomial
		\begin{equation}
			P_{n}(\mu) := \mathrm{det}(\mathbf{B}_{n} - \mu \mathbf{I}_{3 \times 3}) =
			\mu^{3} - c \nu_{n}^{2} \mu^{2} + (a + bd) \nu_{n}^{2} \mu - ac \nu_{n}^{4} \text{ for } \mu \in \mathbb{C}.
			\label{EQUATION_CHARACTERISTIC_POLYNOM_A_N}
		\end{equation}
		Since the matrix
		\begin{equation}
			\mathbf{B}_{n} :=
			\begin{pmatrix}
				1 & 0            & 0 \\
				0 & \tfrac{1}{a} & 0 \\
				0 & 0            & \tfrac{b}{d}
			\end{pmatrix}
			\begin{pmatrix}
				 0             & a \nu_{n} & -b \nu_{n} \\
				-a \nu_{n} & 0             &  0 \\
				 b \nu_{n} & 0             &  \tfrac{cd}{b} \nu_{n}^{2}
			\end{pmatrix} \notag
		\end{equation}
		has real components and is skew-symmetrizable, is has to possess one real and two complex-conjugate eigenvalues.
		Thus, introducing the expressions
		\begin{equation}
			\Delta_{0} = c^{2} \nu_{n}^{4} - 3 (a + bd) \nu_{n}^{2}, \quad
			\Delta_{1} = -2 c^{3} \nu_{n}^{6} + 9 c (a + bd) \nu_{n}^{4} - 27 ac \nu_{n}^{4}, \quad
			C = \sqrt[3]{\tfrac{1}{2} \big(\Delta_{1} + \sqrt{\Delta_{1}^{2} - 4 \Delta_{0}^{3}}\big)}, \notag
		\end{equation}
		we obtain the three roots $\mu_{n, 1}, \mu_{n, 2}, \mu_{n, 3}$ of $\tilde{P}_{n}$ (cf. \cite[p. 179]{PrVe1992})
		\begin{equation}
			\mu_{n, k} = \tfrac{1}{3} \big(c \nu_{n}^{2} - C e^{2 i k \pi/3} - e^{-2 i k \pi/3} \tfrac{\Delta_{0}}{C}\big) \notag
		\end{equation}
		where $\sqrt{\cdot}$ and $\sqrt[3]{\cdot}$
		stand for the main branch of complex square and cubic roots.
	\end{proof}
	
	\begin{lemma}
		Eigenvectors $v_{n, k}$, $k = 0, 1, 2$, of $\mathbf{B}_{n}$ corresponding
		to the eigenvalues $\mu_{n, k}$ of $\mathbf{B}_{n}$ from Lemma \ref{LEMMA_EIGENVALUES_OF_A_N} are given by
		\begin{equation}
			\mathbf{v}_{n, k} =
			\left\{
			\begin{array}{cc}
				\mathbf{e}_{k}, & \text{if } n = 0, \\
				\begin{pmatrix}
					-b \nu_{n} \mu_{n, k} \\
					b \nu_{n}^{2} \\
					a\nu_{n}^{2} + \mu_{n, k}^{2}
				\end{pmatrix}, & \text{otherwise}. \\
			\end{array}\right.
		\end{equation}
		with
		$\mathbf{e}_{1} = (1, 0, 0)^{T}$, $\mathbf{e}_{2} = (0, 1, 0)^{T}$, $\mathbf{e}_{3} = (0, 0, 1)^{T}$.
	\end{lemma}
	
	\begin{proof}
		Since the first case $n = 0$ is obvious, we only consider the case $n > 1$.
		For $k \in \{0, 1, 2\}$, we consider the matrix
		\begin{equation}
			\mu_{n, k} \mathbf{I}_{3 \times 3} - \mathbf{B}_{n} =
			\begin{pmatrix}
				 \mu_{n, k}    & -a \nu_{n} & b \nu_{n} \\
				 \nu_{n}   &  \mu_{n, k}    & 0 \\
				-d \nu_{n} &  0             & \mu_{n, k} - c \nu_{n}^{2}
			\end{pmatrix}. \notag
		\end{equation}
		The latter is singular since $\alpha_{n, k}$ is an eigenvalue of $\mathbf{B}_{n}$.
		Further, due to the fact
		\begin{equation}
			\mathrm{det}(\mathbf{B}_{n}) = ac \nu_{n}^{4} > 0, \notag
		\end{equation}
		$\mathbf{B}_{n}$ is invertible and, therefore, $\mu_{n, k} \neq 0$.
		We want to find a nontrivial vector $\mathbf{v}_{n, k} \in \mathbb{R}^{3}$ satisfying
		\begin{equation}
			\big(\alpha_{n, k} \mathbf{I}_{3 \times 3} - \mathbf{B}_{n}\big) v_{n, k} = \mathbf{0}_{3 \times 1}.
			\label{EQUATION_EIGENVECTOR}
		\end{equation}
		Thus, we can apply a Gauss-Jordan iteration to the former matrix and find
		\begin{equation}
			\mu_{n, k} \mathbf{I}_{3 \times 3} - \mathbf{B}_{n} \sim
			\begin{pmatrix}
				\mu_{n, k} & -a \nu_{n}                      &  b \nu_{n} \\
				0          &  \mu_{n, k}^{2} - a \nu_{n}^{2} & -b \nu_{n}^{2} \\
				0          & -ad \nu_{n}^{2}                 &  \mu_{n, k}^{2} - c \nu_{n}^{2} \mu_{n, k} + bd \nu_{n}^{2}
			\end{pmatrix}.
			\notag
		\end{equation}
		Since the latter matrix must be singular, the third row must be proportional to the second one.
		Thus, Equation (\ref{EQUATION_EIGENVECTOR}) is equivalent with
		\begin{equation}
			\begin{pmatrix}
				\mu_{n, k} & -a \nu_{n}                     &  b \nu_{n} \\
				0          & \mu_{n, k}^{2} + a \nu_{n}^{2} & -b \nu_{n}^{2}
			\end{pmatrix} \mathbf{v}_{n, k} = \mathbf{0}_{3 \times 1}.
			\notag
		\end{equation}
		Since the rank of this matrix is 2, the equation above yields only one eigenvector
		\begin{equation}
			\mathbf{v}_{n, k} =
			\begin{pmatrix}
				-b \nu_{n} \mu_{n, k} \\
				b \nu_{n}^{2} \\
				a \nu_{n}^{2} + \mu_{n, k}^{2}
			\end{pmatrix}
			\notag
		\end{equation}
		being determined up to a multiplicative constant.
	\end{proof}
	
	Note that $\mathbf{v}_{n, 1}, \mathbf{v}_{n, 2}, \mathbf{v}_{n, 3}$ are linearly independent,
	but, in general, not orthonormal.

	Letting now
	\begin{equation}
		\mathbf{D}_{n} := \mathrm{diag}(\mu_{n, 1}, \mu_{n, 2}, \mu_{n, 3}), \notag
	\end{equation}
	we obtain a singular value decomposition for $\mathbf{B}_{n}$
	\begin{equation}
		\mathbf{B}_{n} = \mathbf{S}_{n} \mathbf{D}_{n} \mathbf{S}_{n}^{-1} \notag
	\end{equation}
	with an invertible matrix
	\begin{equation}
		\mathbf{S}_{n} = (\mathbf{v}_{n, 1} \; \mathbf{v}_{n, 2} \; \mathbf{v}_{n, 3})^{T}. \notag
	\end{equation}
	
	Exploiting now Corollary \ref{COROLLARY_DELAYED_EXPONENTIAL_DIAGONALIZED} from Appendix,
	Equation (\ref{EQUATION_SECTION_EXPLICIT_SOLUTION_N_COEFFICIENT_SOLUTION}) can finally be written as
	\begin{equation}
		\mathbf{V}_{n}(t) = \left\{
		\begin{array}{cl}
			\mathbf{V}^{0}_{\tau, n}(t), & t \in [-\tau, 0), \\
			\mathbf{V}^{0}_{n}, & t = 0, \\
			{\mathbf{S} \exp_{\tau}(-\mathbf{D}_{n}, t - \tau) \mathbf{S}^{-1} V^{0}_{n} -
			\mathbf{S} \mathbf{D}_{n} \int_{-\tau}^{0} \exp_{\tau}(-\mathbf{D}_{n}, t - 2\tau - s) \mathbf{S}^{-1} \mathbf{V}^{0}_{\tau, n}(s) \mathrm{d}s + \atop
			\int_{0}^{t} \mathbf{S} \exp_{\tau}(-\mathbf{D}_{n}, t - \tau - s) \mathbf{S}^{-1} \mathbf{F}_{n}(s) \mathrm{d}s}, & t \geq 0
		\end{array}\right.
		\notag
	\end{equation}
	where the inverse of $\mathbf{S}_{n}$ is given by the Laplace formula
	\begin{equation}
		\mathbf{S}^{-1} =
		\frac{
		\begin{pmatrix}
			\phantom{-}S_{n}^{22} S_{n}^{33} - S_{n}^{23} S_{n}^{32} & -S_{n}^{12} S_{n}^{33} + S_{n}^{13} S_{n}^{32}           & \phantom{-}S_{n}^{12} S_{n}^{23} - S_{n}^{13} S_{n}^{22} \\
			-S_{n}^{21} S_{n}^{33} + S_{n}^{23} S_{n}^{31}           & \phantom{-}S_{n}^{11} S_{n}^{33} - S_{n}^{13} S_{n}^{31} & -S_{n}^{11} S_{n}^{23} + S_{n}^{13} S_{n}^{21} \\
			\phantom{-}S_{n}^{21} S_{n}^{32} - S_{n}^{22} S_{n}^{32} & -S_{n}^{11} S_{n}^{32} + S_{n}^{12} S_{n}^{31}           & \phantom{-}S_{n}^{11} S_{n}^{22} - S_{n}^{12} S_{n}^{21}
		\end{pmatrix}}
		{S_{n}^{11} S_{n}^{22} S_{n}^{33} + S_{n}^{12} S_{n}^{23} S_{n}^{31} + S_{n}^{13} S_{n}^{21} S_{n}^{32} -
		 S_{n}^{31} S_{n}^{22} S_{n}^{13} - S_{n}^{32} S_{n}^{23} S_{n}^{11} - S_{n}^{33} S_{n}^{21} S_{n}^{12}}. \notag
	\end{equation}

\begin{appendix}
	\addcontentsline{toc}{section}{Appendix: Delayed Exponential Function}
	\section*{Appendix: Delayed Exponential Function}
	Let $X$ be a real or a complex Hilbert space and let $L(X)$ denote the space of bounded linear operator on $X$.
	For $\tau > 0$ and $\mathcal{B} \in L(X)$, we consider first the following scalar ordinary delay differential equation
	\begin{equation}
		\begin{split}
			\partial_{t} u(t) &= \mathcal{B} u(t - \tau) + f(t) \text{ for } t > 0, \\
			u(0) &= u^{0}, \\
			u(t) &= u^{0}_{\tau} \text{ for } t \in (-\tau, 0).
		\end{split}
		\label{EQUATION_ORDINARY_DELAY_DE}
	\end{equation}
	for some $u^{0} \in X$, $u^{0}_{\tau} \in L^{2}(-\tau, 0; X)$ and $f \in L^{2}_{\mathrm{loc}}(0, \infty; X)$.

	Following the approach in \cite{KhuIvKo2009},
	we introduce the delayed exponential function
	\begin{equation}
		\begin{split}
			\exp_{\tau}(\mathcal{B}, \cdot) &\colon \mathbb{R} \to L(X), \\
			\exp_{\tau}(\mathcal{B}, t) &:=
			\begin{cases}
				0_{L(X)}, & t < -\tau, \\
				\mathrm{id}_{X} + \sum\limits_{k = 1}^{\left\lfloor \tfrac{t}{\tau}\right\rfloor + 1} \frac{(t - (k - 1) \tau)^{k}}{k!} \mathcal{B}^{k}, & t \geq -\tau.
			\end{cases}
		\end{split}
		\notag
	\end{equation}
	Figure \ref{FIGURE_DELAYED_EXPONENTIAL_FUNCTION} displays the delayed exponential function for
	the case that $\mathcal{B}$ is a real number.
	\begin{figure}[h!]
		\centering
		\includegraphics[scale = 0.5]{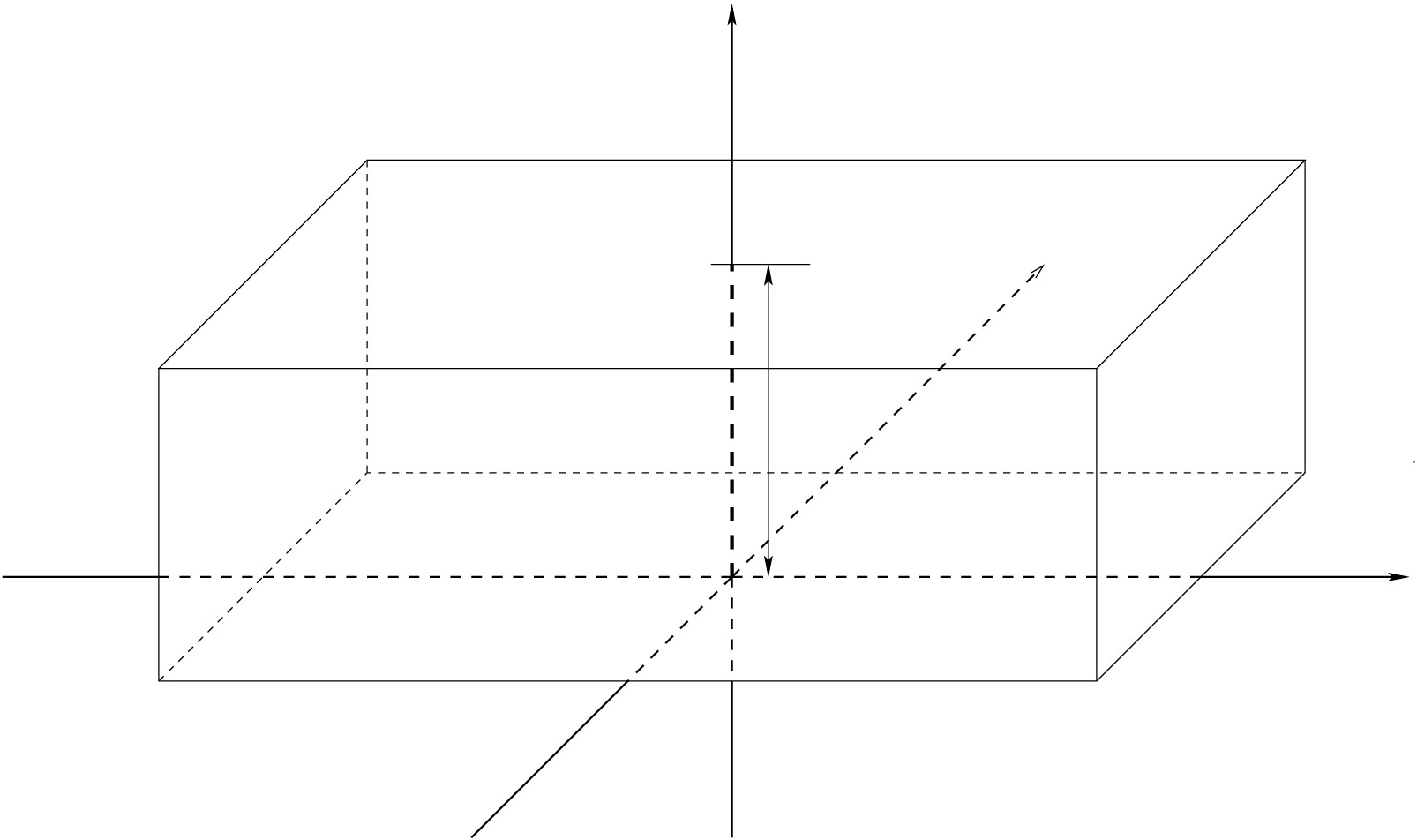}
		\caption{Delayed exponential function. \label{FIGURE_DELAYED_EXPONENTIAL_FUNCTION}}
	\end{figure}
	
	Since $\exp_{\tau}(\mathcal{B}, t)$ is an operator polynomial in $\mathcal{B}$ piecewise with respect to $t$, we obviously have the following representation.
	\begin{theorem}
		Let $\mathcal{S} \colon X \to X$ be an isomorphism, i.e., $\mathcal{S}, \mathcal{S}^{-1} \in L(X)$.
		Then
		\begin{equation}
			\exp_{\tau}(\mathcal{B}, t) = \mathcal{S} \exp_{\tau}(\mathcal{S}^{-1} \mathcal{B} \mathcal{S}, t) \mathcal{S}^{-1} \text{ for } t \in \mathbb{R}. \notag
		\end{equation}
	\end{theorem}
	
	\begin{corollary}
		\label{COROLLARY_DELAYED_EXPONENTIAL_DIAGONALIZED}
		If $X \in \{\mathbb{R}^{d}, \mathbb{C}^{d}\}$, $d \in \mathbb{N}$, 
		and $\mathcal{B} = \mathbf{B}$ is diagonalizable over $\mathbb{C}$, i.e.,
		if there exists a diagonal matrix $\mathbf{D} = \mathrm{diag}(\lambda_{1}, \dots, \lambda_{n})$,
		$\lambda_{1}, \dots, \lambda_{d} \in \mathbb{C}$,
		and an invertible $\mathbf{S} \in \mathbb{C}^{d \times d}$ such that
		$\mathbf{A} = \mathbf{S} \mathbf{D} \mathbf{S}^{-1}$, then
		\begin{equation}
			\exp_{\tau}(\mathbf{B}, t) = \mathbf{S} \exp_{\tau}(\mathbf{D}, t) \mathbf{S}^{-1} =
			\mathbf{S} \, \mathrm{diag}\big(\exp_{\tau}(\lambda_{1}, t), \dots, \exp_{\tau}(\lambda_{n}, t)\big) \mathbf{S}^{-1}
			\text{ for } t \in \mathbb{R}. \notag
		\end{equation}
	\end{corollary}
	
	According to \cite[Theorem 3.12]{KhuPoRa2013}, we have the following well-posedness result for Equation (\ref{EQUATION_ORDINARY_DELAY_DE}).
	\begin{theorem}
		\label{THEOREM_DELAY_ODE}
		The delay differential equation (\ref{EQUATION_ORDINARY_DELAY_DE}) possesses a unique strong
		$u \in L^{2}_{\mathrm{loc}}(-\tau, \infty; X) \cap H^{1}_{\mathrm{loc}}(0, \infty; X)$ given by
		\begin{equation}
			u(t) = \left\{
			\begin{array}{cl}
				\varphi(t), & t \in [-\tau, 0), \\
				u^{0}, & t = 0, \\
				{\exp_{\tau}(-\mathcal{B}, t - \tau) u^{0} -
				\mathcal{B} \int_{-\tau}^{0} \exp_{\tau}(-\mathcal{B}, t - 2\tau - s) u^{0}_{\tau}(s) \mathrm{d}s + \atop
				\int_{0}^{t} \exp_{\tau}(-\mathcal{B}, t - \tau - s) f(s) \mathrm{d}s}, & t \geq 0.
			\end{array}\right.
			\label{EQUATION_SOLUTION_DELAY_ODE}
		\end{equation}
		If $u^{0}_{\tau}$ lies in $C^{0}\big([-\tau, 0], X\big)$ and satisfies the compatibility condition $u^{0}_{\tau}(0) = u^{0}$,
		then the strong solution is even a classical solution, i.e.,
		$u \in C^{0}\big([-\tau, \infty), X\big) \cap C^{1}\big([0, \infty), X\big)$.
	\end{theorem}
\end{appendix}

\section*{Conflict of Interests}
The authors declare that there is no conflict of interests regarding the publication of this paper.

\section*{Acknowledgment}
The investigations were supported by the Young Scholar Fund
(Research Grant ZUK 52/2 of the Deutsche Forschungsgemeinschaft) at the University of Konstanz, Konstanz, Germany

\end{document}